\documentclass[14pt]{amsart}

\usepackage{cases}
\usepackage{amsmath}
\usepackage{amsfonts}
\usepackage{bm}
\usepackage{amsfonts,amsmath,amssymb,amscd,bbm,amsthm,mathrsfs,dsfont}
\usepackage{mathrsfs}
\usepackage{pb-diagram}
\usepackage{amssymb}
\usepackage{xypic}
\usepackage[dvips]{graphicx}
\usepackage[all]{xy}
\usepackage{CJK}
\usepackage[dvipsnames]{xcolor}

\newtheorem{Theorem}{Theorem}[section]
\newtheorem{Lemma}[Theorem]{Lemma}
\newtheorem{Definition}[Theorem]{Definition}
\newtheorem{Corollary}[Theorem]{Corollary}
\newtheorem{Proposition}[Theorem]{Proposition}
\newtheorem{Example}[Theorem]{Example}
\newtheorem{Remark}[Theorem]{Remark}
\newtheorem{Conjecture}[Theorem]{Conjecture}

\newtheorem{Question}[Theorem]{Question}
\date{version of \today}

 \normalbaselines
\title [Some conjectures on generalized cluster algebras ]
{Some conjectures on generalized cluster algebras via\\ the cluster formula and $D$-matrix pattern }

\author{Peigen Cao $\;\;\;\;\;\;$ Fang Li $\;\;\;\;\;\;$}
\address{Peigen Cao
\newline Department
of Mathematics, Zhejiang University (Yuquan Campus), Hangzhou, Zhejiang
310027,  P.R.China}
\email{peigencao@126.com}

\address{Fang Li
\newline Department
of Mathematics, Zhejiang University (Yuquan Campus), Hangzhou, Zhejiang
310027, P.R.China}
\email{fangli@zju.edu.cn}

\begin{document}
\begin{CJK*}{GBK}{song}
\renewcommand{\thefootnote}{\alph{footnote}}

\begin{abstract}
In the theory of generalized cluster algebras, we build the so-called cluster formula and $D$-matrix pattern. Then as applications, some fundamental conjectures of generalized cluster algebras are solved affirmatively.

\end{abstract}

\maketitle
\bigskip

\section{introduction}

Cluster algebras were introduced by Fomin and Zelevinsky in \cite{FZ}. The motivation was to create a common framework for phenomena occurring in connection
with total positivity and canonical bases. By now numerous connections between cluster algebras and other branches of mathematics have been discovered, e.g.
the theory of quiver representations, categorifications over some important algebras  and Poisson geometry, etc..

There are many interesting conjectures about cluster algebras, for example, as follows. Note that in this paper, {\em the positive integer $n$ always denotes the rank of a cluster algebra.}

\begin{Conjecture}
\label{conj}
(\cite{FZ2,GSV})(a) The exchange graph of a cluster algebra with rank $n$ only depends on the initial exchange matrix;

(b) Every seed is uniquely determined by its cluster under mutation equivalence;

(c) Two clusters are adjacent in the exchange graph if and only if they
have exactly $n-1$ common cluster variables.

\end{Conjecture}

In \cite{GSV}, M. Gekhtman, M. Shapiro and A. Vainshtein proved the following facts in the skew-symmetrizable case for standard cluster algebras:

(1) $(a)$ is true for $B$ with full rank.

(2) $(b)$ implies $(c)$.

(3) $(b)$ is true for cluster algebras of geometric type, and for cluster algebras whose exchange matrix is of full rank.

It is also known that $(b)$ is true for cluster algebras having some ``realization", for example cluster algebras from surfaces \cite{FST} and the cluster algebra which has a categorification \cite{BMRRT,DL,AC}.

We know a cluster should contain the whole information of the corresponding exchange matrix under the assumption that $(b)$ is true.
Trivially,  $(b)$ implies the following statement $(d)$:

$(d)$ The exchange matrix could be uniquely recovered from a given cluster.

 In this paper, our aim is to discuss the above conjectures for generalized cluster algebras.

Generalized cluster algebras were introduced in \cite{CS} by Chekhov and Shapiro,
which are the generalization of the (standard) cluster algebras introduced by Fomin and Zelevinsky
in \cite{FZ}. In the standard case,  a product of cluster variables, one known and one
unknown, is equal to a binomial in other known variables.  These binomial exchange relations is replaced by polynomial exchange relations in generalized cluster algebras. The structure of generalized cluster algebras naturally appears from the Teichmuller spaces of Riemann surfaces with orbifold points \cite{CS}. It also  is raised  in representations of quantum affine algebras \cite{G} and in WKB analysis \cite{IN}.
It can be seen in \cite{CS,NT} that
many important properties and definitions of the standard cluster algebras are naturally
extended to the generalized ones, for examples, Laurent phenomenon, finite type classification, the ${\bf c}$-vectors, ${\bf g}$-vectors and {\bf$F$-polynomials}. From these views, we know generalized cluster algebra is an essential improvement of the standard cluster algebras.

We want to consider  Conjecture \ref{conj}  in the case of generalized cluster algebra. As a tool for studying generalized cluster algebras including cluster algebras, we give the so-called {\bf cluster formula}. Relying the using of the cluster formula, our method will be constructive, in particular, to recover the exchange matrix from a given cluster, as a direct proof of $(d)$. Finally,  we will show that Conjecture  \ref{conj} holds  for generalized cluster algebra in general case.

Cluster algebras is introduced in a general case, which means their coefficients are in a general semifield. But it seems that many researchers are more interested in cluster algebras of geometric type.  Many conjectures were proved for such cluster algebras, for example, Conjecture  \ref{conj} given above. However, these conjectures are also believed true for cluster algebras with general coefficients. In order to consider Conjecture  \ref{conj} in the case of cluster algebras with general coefficients, we introduce the {\bf$D$-matrix pattern}, which explains the connection between any two (generalized) cluster algebras having the same initial exchange matrix with different coefficient rings from the view of exchange graphs. More precisely, we give a positive answer to Conjecture  \ref{conj} for (generalized) cluster algebras with any coefficients, whose restricted results on standard cluster algebras are also an improvement of the conclusions in the case of geometrical type, given by other mathematicians early.

 This paper is organized as follows: in Section 2, some basic definitions are needed. In Section 3, we give the cluster formula, which is a main result in this papper. As an application, we prove $(b)$ in Conjecture \ref{conj} is true for generalized cluster pattern of weak geometric type. In the final part of the Section 3, we give the connection between cluster formula and compatible 2-form. In Section 4, we give a positive answer to to Conjecture \ref{conj} in the case of generalized cluster algebra.

\section{Preliminaries}

We know that $(\mathbb P, \oplus, \cdot)$ is a {\bf semifield } if $(\mathbb P,  \cdot)$ is an abelian multiplicative group endowed with a binary operation of auxiliary addition $\oplus$ which is commutative, associative, and distributive with respect to the multiplication $\cdot$ in $\mathbb P$.

Let $Trop(u_i:i\in I)$ be a free abelian group generated by $\{u_i: i\in I\}$ for a finite set of index $I$. We define the addition $\oplus$ in $Trop(u_i:i\in I)$ by $\prod\limits_{i}u_i^{a_i}\oplus\prod\limits_{i}u_i^{b_i}=\prod\limits_{i}u_i^{min(a_i,b_i)}$, then $(Trop(u_i:i\in I), \oplus)$ is a semifield, which is called a {\bf tropical semifield}.

The multiplicative group of any semifield $\mathbb P$ is torsion-free for multiplication \cite{FZ}, hence its group ring $\mathbb Z\mathbb P$ is a domain.

The following proposition can be checked directly:
\begin{Proposition}
Assume  $\mathbb P_1,\mathbb P_2$ are two semifield, let $\mathbb P=\mathbb P_1\amalg\mathbb P_2=\{(p_1,p_2)|p_1\in\mathbb P_1,p_2\in\mathbb P_2\}$. Then $\mathbb P$ is a semifield via $(p_1,p_2)\cdot(\bar p_1,\bar p_2):=(p_1\cdot\bar p_1,p_2\cdot\bar p_2)$ and $(p_1,p_2)\oplus(\bar p_1,\bar p_2):=(p_1\oplus\bar p_1,p_2\oplus\bar p_2)$.
\end{Proposition}

\begin{Definition} (i)~
A square integer matrix $B=(b_{ij})_{n\times n}$ is called {\bf skew-symmetric} if $b_{ij}=-b_{ji}$ for any $i,j=1,\cdots,n$; (ii)~
 In general, $B=(b_{ij})_{n\times n}$ is called {\bf skew-symmetrizable} if there exists a diagonal matrix $T$ with positive integer diagonal entries $s_1,\cdots,s_n$ such that $TB$ is skew-symmetric.

\end{Definition}

We take an ambient field $\mathcal F$  to be the field of rational functions in $n$ independent variables with coefficients in $\mathbb Z\mathbb P$.

\begin{Definition} \cite{FZ, FZ2, NT}~
A {\bf(labeled) seed} $\Sigma$ in $\mathcal F$ is a triplet $(X,Y,B)$ such that

(i)  $X=(x_1,\cdots, x_n)$ is an $n$-tuple with $n$ algebraically independent variables $x_1\cdots,x_n$ over $\mathbb {ZP}$. We call $X$ a {\bf cluster} and $x_1\cdots,x_n$ {\bf cluster variables}.

(ii) $Y=(y_1,\cdots,y_n)$ is an $n$-tuple of elements in $\mathbb P$, where $y_1,\cdots,y_n$ are called {\bf coefficents}.

(iii) $B=(b_{ij})$ is an $n\times n$ integer skew-symmetrizable matrix, called an {\bf exchange matrix}.

\end{Definition}

Let $(R,Z)$ be a pair with  $R=(r_i)_{n\times n}$ a diagonal matrix, $r_i\in \mathds{N}$, and  $Z=(z_{i,m})_{i=1,\cdots,n;~m=1,\cdots,r_i-1}$ a family of elements in $\mathbb P$ satisfying the reciprocity condition
$z_{i,m}=z_{i,r_i-m}$ for $m=1,\cdots,r_i-1$. And, denote the notations $z_{i,0}=z_{i,r_i}=1$ for $i=1,\cdots,n$.

\begin{Definition}(\cite{NT})
Let $\Sigma=(X,Y,B)$ be a seed in $\mathcal F$, we define the {\bf $(R,Z)$-mutation} $\mu_k(\Sigma)=\bar \Sigma=(\bar X, \bar Y, \bar B)$ of $\Sigma$ in the direction $k\in\{1,\cdots,n\}$ as a new seed in $\mathcal F$:
\begin{eqnarray}
\label{eq1}\bar x_i&=&\begin{cases}x_i~,&\text{if } i\neq k\\ x_k^{-1}\left(\prod\limits_{j=1}^nx_j^{[-b_{jk}]_+}\right)^{r_k}(\sum\limits_{m=0}^{r_k}z_{k,m}\hat y_k^m)/(\bigoplus\limits_{m=0}^{r_k}z_{k,m} y_k^m),~& \text{if }i=k.\end{cases}\\
\label{eq2}\bar y_i&=&\begin{cases} y_k^{-1}~,& i=k\\ y_i\left(y_k^{[b_{ki}]_+}\right)^{r_k}\left(\bigoplus\limits_{m=0}^{r_k}z_{k,m} y_k^m\right)^{-b_{ki}}~,&otherwise.
\end{cases}\\
\label{eq3}\bar b_{ij}&=&\begin{cases} -b_{ij}~,& i=k\text{ or } j=k\\b_{ij}+r_k(b_{ik}[-b_{kj}]_++[b_{ik}]_+b_{kj})~,&otherwise.\end{cases}
\end{eqnarray}
for $i,j=1,2,\cdots,n$,
where $[a]_+=max\{a,0\},~\hat y_i=y_i\prod\limits_{j=1}^nx_j^{b_{ji}}$.

\end{Definition}
\begin{Remark}\label{rmkb}
(i). It is easy to check that the $(R,Z)$-mutation $\mu_k$ is an involution.

(ii). If $R=I_n$, then (\ref{eq3}) is the standard matrix mutation. Let $B^{\prime}$ be the matrix obtained from $BR$ by the standard matrix mutation in the direction $k$, it is easy to see $B^{\prime}=\bar BR$. We can write $\mu_k(BR)=\mu_k^g(B)R$, where $\mu_k^g(B)=\bar B$ with $\mu_k^g$ called the {\bf generalized matrix mutation}.
\end{Remark}

\begin{Definition}
(\cite{NT})  An {\bf $(R,Z)$-cluster pattern} (or say, {\bf generalized cluster pattern}) $M$ in $\mathcal F$ is an assignment for each seed  $\Sigma_t$ to a vertex $t$ of the $n$-regular tree $\mathbb T_n$, such that for any edge $t^{~\underline{\quad k \quad}}~ t^{\prime},~\Sigma_{t^{\prime}}=\mu_k(\Sigma_t)$. The triple of $\Sigma_t$ are written as follows:$$X_t=(x_{1;t},\cdots, x_{n;t}),~ Y_t=(y_{1;t},\cdots, y_{n;t}), ~B_t=(b_{ij}^t).$$
\end{Definition}

\begin{Remark}
(i)~ Clearly, for each vertex of $\mathbb T_n$, we can uniquely determine the $(R,Z)$-cluster pattern under $(R,Z)$-mutations.

(ii)~ When $R=I_n$ the identity matrix, $Z$ must be empty. In this case, the generalized cluster pattern is just the standard cluster pattern.

\end{Remark}

\begin{Definition} Let $M$ be an $(R,Z)$-cluster pattern, we denote by $\mathcal X=\{x_{i;t}:t\in \mathbb T_n, 1\leq i\leq n\}$ the set of all cluster variables.  The {\bf generalized cluster algebra} $\mathcal A$ associated with a given $(R,Z)$-cluster pattern is the $\mathbb {ZP}$-subalgebra of the field $\mathcal F$ generated by all cluster variables, i.e. $\mathcal A=\mathbb {ZP}[\mathcal X]$. \end{Definition}

By definition, $\mathcal A$ can be obtained from any given seed $\Sigma_{t_0}$ for $t_0\in\mathbb T_n$ via mutations. So, we denote $\mathcal A=\mathcal A(\Sigma_{t_0})$ and call $\Sigma_{t_0}$ the {\bf initial seed} of $\mathcal A$.

\begin{Definition}
(Restriction)
(i) Let $J$ be a subset of $\langle n\rangle=\{1,2,\cdots,n\}$. Remove from $\mathbb T_n$ all edges labeled by indices in $\langle n\rangle\backslash J$, and denote $\mathbb T_n^{t_0}(J)$ the connected component of the resulting graph containing the vertex $t_0$ in  $\mathbb T_n$. Say that $\mathbb T_n^{t_0}(J)$ is obtained from $\mathbb T_n$ by restriction to $J$ including $t_0$. Trivially, $\mathbb T_n^{t_0}(J)$ is a $|J|$-regular tree.

(ii) Let $M$ be an $(R,Z)$-cluster pattern on $\mathbb T_n$ in $\mathcal F$ with the seed $\Sigma_t=(X_t,Y_t,B_t)$. We define a {\bf restricted generalized cluster pattern} $M^{t_0}(J)$ on $\mathbb T_n^{t_0}(J)$ by assigning the seed $\Sigma_t^{t_0}(J)=(\bar X_{t},\bar Y_{t},\bar B_{t})$ at $t\in\mathbb T_n^{t_0}(J)$ with
$\bar X_{t}=(x_{j;t})_{j\in J}, ~\bar Y_{t}=(y_{j;t}\prod\limits_{i\in \langle n\rangle\backslash J}x_{i;t_0}^{b_{ij}^{t}})_{j\in J}, ~\bar B_{t}=(b_{ij}^{t})_{i,j\in J}$. Actually, $M^{t_0}(J)$ is a $(\bar R,\bar Z)$-cluster pattern on $\mathbb T_n^{t_0}(J)$ in the semifield  $\mathbb P\amalg Trop(x_{i;t_0}:i\in\langle n\rangle\backslash J)$, where $\bar R=diag\{r_j\}_{j\in J},~\bar Z=(z_{j,m})_{j\in J;~m=1,\cdots,r_j-1}$. Say that $M^{t_0}(J)$  is obtained from $M$ by restriction  to $J$ including $t_0$.
\end{Definition}

 Note that in the view in \cite{HL}, we can think the seed $\Sigma_t^{t_0}(J)$ as a mixing-type subseed of $\Sigma_t$, that is, $\Sigma_t^{t_0}(J)=(\Sigma_t)_{\langle n\rangle\backslash J,\emptyset}$.

\begin{Definition}
$(a)$ Assume $\mathbb P_1$ is a semifield, $\mathbb P_2=Trop(u_j:j\in I)$, where $h=|I|<+\infty$.  An $(R,Z)$-cluster pattern $M$ with coefficients in $\mathbb P=\mathbb P_1\coprod \mathbb P_2$ is said to be of {\bf weakly geometric type} if the following hold:

(i) $Z$ is a family of elements in $\mathbb P_1$.

(ii) $y_{i;t}$ is a Laurent monomial and denote it by $y_{i;t}=u_1^{c_{1i}^t}u_2^{c_{2i}^t}\cdots u_h^{c_{hi}^t}$.

$(b)$ Further, if $\mathbb P_1=Trop(Z)$, where we regard  $Z=(z_{i,m})_{i=1,\cdots,n;m=1,\cdots,r_i-1}$ with $z_{i,m}=z_{i,r_i-s}$ as formal variables, then we say $M$ to be an $(R,Z)$-{\bf cluster pattern of geometric type}.
\end{Definition}

\begin{Proposition}
\label{2pro1}
Let $M$ be a  $(R,Z)$-cluster pattern of weakly geometric type,  $Y_t=(y_{1;t},\cdots,y_{n;t})$ be the coefficient tuples at $t$, where $y_{i;t}=u_1^{c_{1i}^t}u_2^{c_{2i}^t}\cdots u_h^{c_{hi}^t}$. Define  $C_t=(c_{ij}^t)$. Then for any edge $t^{~\underline{\quad k \quad}}~ \bar t$ in $\mathbb T_n$, $C_t$ and $C_{\bar t}$ are related by the {\bf formula of mutation of $C$-matrices}:$$\bar c_{ij}^t=\begin{cases}-c_{ij}^t,~&\text{if } j=k;\\ c_{ij}^t+r_k(c_{ik}^t[b_{kj}^t]_++[-c_{ik}^t]_+b_{kj}^t),&otherwise.\end{cases}$$
\end{Proposition}

\begin{proof}
 By (\ref{eq2}), we have $$u_1^{\bar c_{1i}^t}\cdots u_h^{\bar c_{hi}^t}=\begin{cases} (u_1^{c_{1k}^t}u_2^{c_{2k}^t}\cdots u_h^{c_{hk}^t})^{-1}~,& \text{if}~ i=k;\\ u_1^{c_{1i}^t}\cdots u_h^{c_{hi}^t}\left((u_1^{c_{1k}^t}\cdots u_h^{c_{hk}^t})^{[b_{ki}]_+}\right)^{r_k}\left(\bigoplus\limits_{m=0}^{r_k}z_{k,m} (u_1^{c_{1k}^t}\cdots u_h^{c_{hk}^t})^m\right)^{-b_{ki}}~,&\text{otherwise}.
\end{cases}$$
Then, we obtain $\bigoplus\limits_{m=0}^{r_k} z_{k,m}=1$ for $k=1,\cdots,n$ and
$$\bar c_{ij}^t=\begin{cases}-c_{ij}^t,&\text{if } j=k;\\ c_{ij}^t+r_k(c_{ik}^t[b_{kj}^t]_++[-c_{ik}^t]_+b_{kj}^t),&otherwise.\end{cases}$$
\end{proof}

\begin{Definition}
 We say $M$ to be an {\bf $(R,Z)$-cluster pattern with (weakly) principle coefficients} at $t_0$, if
$M$ is of (weakly) geometric type on $\mathbb T_n$ and $C_{t_0}=I_n$.
\end{Definition}
\begin{Remark}
The definition of $(R,Z)$-cluster pattern with  principle coefficients  given here is the same with the one in \cite{NT}.
\end{Remark}

\section{Cluster formula and related results}

\subsection{The cluster formula}.

 A fundamental fact is that for any vertex $t\in \mathbb T_n$ and the corresponding seed $\Sigma_t=(X_t,Y_t,B_t)$, $RB_t$ is always a skew-symmetrizable matrix, that is, there is a positive integer diagonal matrix $S$ such that $SRB_t$ is skew-symmetric.

 Indeed, since $B_t$ is  skew-symmetrizable, we have a positive integer diagonal matrix $T$ which does not depend on $t$ such that $TB_t$ is a skew-symmetric matrix.
For $R=(r_i)$, let $S=(\prod\limits_{i=1}^nr_i)TR^{-1}$. Trivially, $S$ is a positive integer diagonal matrix. Then $SRB_t=(\prod\limits_{i=1}^nr_i)TR^{-1}RB_t=(\prod\limits_{i=1}^nr_i)TB_t$ is skew-symmetric.

 The diagonal matrix $S=(\prod\limits_{i=1}^nr_i)TR^{-1}$ will be valuable for the following discussion, which is called the {\bf $R$-skew-balance} for all seeds $\Sigma_t$ with $t\in \mathbb T_n$.

  Let $\Sigma_t=( X_t,  Y_t, B_t),~\Sigma_{t_0}=( X_{t_0},  Y_{t_0}, B_{t_0})$ be two seeds of $M$ at $t$ and $t_0$. Considering $\Sigma_{t_0}$ as the initial seed, we know $x_{i;t}$ is a rational function in $x_{1;t_0},\cdots,x_{n;t_0}$ with coefficients in $\mathbb Z\mathbb P$ for each $i$.

Let $$J^t_{t_0}(X)=\begin{pmatrix} \frac{\partial x_{1;t}}{\partial x_{1;t_0}}&\frac{\partial x_{2;t}}{\partial x_{1;t_0}}&\cdots &\frac{\partial x_{n;t}}{\partial x_{1;t_0}}\\ \frac{\partial x_{1;t}}{\partial x_{2;t_0}}&\frac{\partial x_{2;t}}{\partial x_{2;t_0}}&\cdots &\frac{\partial x_{n;t}}{\partial x_{2;t_0}}\\ \vdots &\vdots& &\vdots\\ \frac{\partial x_{1;t}}{\partial x_{n;t_0}}&\frac{\partial x_{2;t}}{\partial x_{n;t_0}}&\cdots &\frac{\partial x_{n;t}}{\partial x_{n;t_0}} \end{pmatrix},\;\; H^t_{t_0}(X)=  \begin{pmatrix} \frac{x_{1;t_0}}{x_{1;t}}\cdot\frac{\partial x_{1;t}}{\partial x_{1;t_0}}&\frac{x_{1;t_0}}{x_{2;t}}\cdot\frac{\partial x_{2;t}}{\partial x_{1;t_0}}&\cdots &\frac{x_{1;t_0}}{x_{n;t}}\cdot\frac{\partial x_{n;t}}{\partial x_{1;t_0}}\\ \frac{x_{2;t_0}}{x_{1;t}}\cdot\frac{\partial x_{1;t}}{\partial x_{2;t_0}}&\frac{x_{2;t_0}}{x_{2;t}}\cdot\frac{\partial x_{2;t}}{\partial x_{2;t_0}}&\cdots &\frac{x_{2;t_0}}{x_{n;t}}\cdot\frac{\partial x_{n;t}}{\partial x_{2;t_0}}\\ \vdots &\vdots& &\vdots\\ \frac{x_{n;t_0}}{x_{1;t}}\cdot\frac{\partial x_{1;t}}{\partial x_{n;t_0}}&\frac{x_{n;t_0}}{x_{2;t}}\cdot\frac{\partial x_{2;t}}{\partial x_{n;t_0}}&\cdots &\frac{x_{n;t_0}}{x_{n;t}}\cdot\frac{\partial x_{n;t}}{\partial x_{n;t_0}} \end{pmatrix},$$ we can obtain
 $H^t_{t_0}(X)=diag(x_{1;t_0},\cdots, x_{n;t_0})J^t_{t_0}(X)diag(\frac{1}{x_{1;t}}, \cdots, \frac{1}{x_{n;t}})$.

\begin{Lemma}
\label{3lem2}$H_{u}^{v}(X)H_{v}^{w}(X)=H_{u}^{w}(X)$ for any $u,v,w\in \mathbb T_n$. In particular, $H_{u}^{v}(X)^{-1}=H_{v}^{u}(X)$.
\end{Lemma}

\begin{proof}
We can view $x_{j;w}$ as a rational function in $x_{1;v},\cdots,x_{n;v}$, and view $x_{k;v}$ as a rational function in $x_{1;u},\cdots,x_{n;u}$, where $j,k=1,\cdots,n$.
Thus $\frac{\partial x_{j;w}}{\partial x_{i;u}}=\sum\limits_{k=1}^n \frac{\partial x_{j;w}}{\partial x_{k;v}}\cdot \frac{\partial x_{k;v}}{\partial x_{i;u}}=\sum\limits_{k=1}^n \frac{\partial x_{k;v}}{\partial x_{i;u}}\cdot \frac{\partial x_{j;w}}{\partial x_{k;v}}$,
 i.e. $J_{u}^{w}(X)=J_{u}^{v}(X)J_{v}^{w}(X)$. It follows that
 \begin{eqnarray}
&&H_{u}^{v}(X)H_{v}^{w}(X)\nonumber\\
 &=&diag\{x_{1;u},\cdots,x_{n;u}\}J_{u}^{v}(X)diag\{\frac{1}{x_{1;v}},\cdots,\frac{1}{x_{n;v}}\}
\cdot diag\{x_{1;v},\cdots,x_{n;v}\}J_{v}^{w}(X)diag\{\frac{1}{x_{w;1}},\cdots,\frac{1}{x_{w;n}}\}\nonumber\\
 &=&diag\{x_{1;u},\cdots,x_{n;u}\}J_{u}^{v}(X)J_{v}^{w}(X)diag\{\frac{1}{x_{1;w}},\cdots,\frac{1}{x_{n;w}}\}\nonumber\\
 &=&diag\{x_{1;u},\cdots,x_{n;u}\}J_{u}^{w}(X)diag\{\frac{1}{x_{1;w}},\cdots,\frac{1}{x_{n;w}}\}\nonumber\\
 &=&H_{u}^{w}(X).\nonumber
 \end{eqnarray}
\end{proof}

\begin{Lemma}
\label{3lem1}
If $\Sigma_v=\mu_k(\Sigma_u)$ for $k\in\langle n\rangle$, then
$H_u^v(X)(B_vR^{-1}S^{-1})H_u^v(X)^\top=B_uR^{-1}S^{-1}$, and $det(H_u^v(X))=-1$.
\end{Lemma}

\begin{proof}
We assume that $\Sigma_u=\Sigma,\Sigma_v=\bar \Sigma,~x_{i;u}=x_i,~x_{i;v}=\bar x_i$ for $i\in\langle n\rangle$ and we denote $H=H_u^v(X)$.
We have
\begin{eqnarray}
\bar x_i&=&\begin{cases}x_i~,&\text{if } i\neq k,\\ x_k^{-1}(\prod\limits_{j=1}^nx_j^{[-b_{jk}]_+})^{r_k}\frac{\sum\limits_{m=0}^{r_k}z_{k,m}\hat y_k^m}{\bigoplus\limits_{m=0}^{r_k}z_{k,m} y_k^m}~,& \text{if }i=k.\end{cases}\nonumber
\end{eqnarray}
For $l=k,~i\neq k$, we have
\begin{eqnarray}
\frac{\partial \bar x_k}{\partial x_i}&=&\frac{x_k^{-1}}{\bigoplus\limits_{m=0}^{r_k}z_{k,m} y_k^m}[\frac{\partial(\prod\limits_{j=1}^nx_j^{[-b_{jk}]_+})^{r_k}}{\partial x_i}\sum\limits_{m=0}^{r_k}z_{k,m}\hat y_k^m+(\prod\limits_{j=1}^nx_j^{[-b_{jk}]_+})^{r_k}\frac{\partial(\sum\limits_{m=0}^{r_k}z_{k,m}\hat y_k^m)}{\partial x_i}]\nonumber\\
&=&r_k[-b_{ik}]_+x_i^{-1}\bar x_k+\frac{x_i^{-1}x_k^{-1}}{\bigoplus\limits_{m=0}^{r_k}z_{k,m} y_k^m}(\prod\limits_{j=1}^nx_j^{[-b_{jk}]_+})^{r_k}\sum\limits_{m=0}^{r_k}mz_{k,m}\hat y_k^mb_{ik}   \nonumber\\
\text{ thus}\hspace{10mm}\frac{x_i}{\bar x_k}\frac{\partial \bar x_k}{\partial x_i}&=&[-b_{ik}]_+r_k+\frac{\sum\limits_{m=0}^{r_k}mz_{k,m}\hat y_k^mb_{ik}}{\sum\limits_{m=0}^{r_k}z_{k,m}\hat y_k^m}.  \nonumber
\end{eqnarray}
We have
\begin{eqnarray}
\label{eq4} H_{il}&=&\frac{x_i}{\bar x_l}\frac{\partial \bar x_l}{\partial x_i} =\begin{cases}\delta_{il}~,& \text{if } l\neq k,\\-1 ~,& \text {if }i=l=k,\\ [-b_{ik}]_+r_k+\frac{\sum\limits_{m=0}^{r_k}mz_{k,m}\hat y_k^mb_{ik}}{\sum\limits_{m=0}^{r_k}z_{k,m}\hat y_k^m}~,&\text {if }i\neq k,~l=k.\end{cases}
\end{eqnarray}

It is easy to see that $det(H)=-1$.
Without lose of generality, we may assume $k=1$.
Let \\$\beta=\begin{pmatrix}\bar b_{21}\\ \bar b_{31}\\ \vdots\\ \bar b_{n1}\end{pmatrix},~ \alpha=\begin{pmatrix}a_2\\a_3\\ \vdots\\a_n\end{pmatrix},$ where $a_i=[-b_{i1}]_+r_1+\frac{\sum\limits_{m=0}^{r_1}mz_{1,m}\hat y_1^mb_{i1}}{\sum\limits_{m=0}^{r_1}z_{1,m}\hat y_1^m}$. Then $H=\begin{pmatrix}-1&0\\ \alpha&I_{n-1}\end{pmatrix}$.

 Let $R=\begin{pmatrix}r_1&0\\ 0&R_1\end{pmatrix}, S=\begin{pmatrix}s_1&0\\ 0&S_1\end{pmatrix},$ where $R_1=diag\{r_2,r_3,\cdots,r_n\}, ~ S_1=diag\{s_2,s_3,\cdots,s_n\}$.  We can write $B_v=\begin{pmatrix}0&\gamma^{\top}\\ \beta&B_1\end{pmatrix}$, where $\gamma^{\top}=-s_1^{-1}r_1^{-1}\beta^{\top}R_1S_1$, due to the fact that $SRB_v$ is skew-symmetric.
So, for $B=B_u$, $\bar B=B_v$,  we need only to show that $B_u=HB_vR^{-1}S^{-1}H^{\top}SR$.

Denote $M=HB_vR^{-1}S^{-1}H^{\top}SR$. Then, $M=$\\
$\begin{pmatrix}-1&0\\ \alpha&I_{n-1}\end{pmatrix}\begin{pmatrix}0&-s_1^{-1}r_1^{-1}\beta^{\top}R_1S_1\\
\beta&B_1\end{pmatrix}\begin{pmatrix}r_1^{-1}&0\\ 0&R_1^{-1}\end{pmatrix}\begin{pmatrix}s_1^{-1}&0\\0&S_1^{-1}\end{pmatrix}
\begin{pmatrix}-1&0\\ \alpha&I_{n-1}\end{pmatrix}^{\top}\begin{pmatrix}s_1&0\\ 0&S_1\end{pmatrix}\begin{pmatrix}r_1&0\\ 0&R_1\end{pmatrix}$\\
$=\begin{pmatrix}0&s_1^{-1}r_1^{-1}\beta^{\top}R_1S_1\\ -\beta  &B_1+s_1^{-1}r_1^{-1}\beta \alpha^{\top}S_1R_1-s_1^{-1}r_1^{-1}\alpha\beta^{\top}S_1R_1\end{pmatrix}$.

Thus, we can obtain by replacing the entries that
 \begin{eqnarray}
M_{il}&=&\begin{cases}-\bar b_{il}~,&\text{if } i=1\text{ or }l=1,\\  \bar b_{il}+r_1^{-1}s_1^{-1}\bar b_{i1}\left([-b_{l1}]_+r_1+\frac{\sum\limits_{m=0}^{r_1}mz_{1,m}\hat y_1^mb_{l1}}{\sum\limits_{m=0}^{r_1}z_{1,m}\hat y_1^m}\right)s_lr_l-\\ \hspace{4mm}r_1^{-1}s_1^{-1}\left([-b_{i1}]_+r_1+\frac{\sum\limits_{m=0}^{r_1}mz_{1,m}\hat y_1^mb_{i1}}{\sum\limits_{m=0}^{r_1}z_{1,m}\hat y_1^m}\right)\bar b_{l1}s_lr_l
~,&\text{otherwise}.
  \end{cases}\nonumber \\
  &=&\begin{cases}-\bar b_{il}~,&\text{if } i=1\text{ or }l=1,\\ \bar b_{il}+s_1^{-1}s_lr_l\bar b_{i1}[-b_{l1}]_+-s_1^{-1}s_lr_l[-b_{i1}]_+\bar b_{l1}~,
&\text{otherwise}.
  \end{cases}\nonumber
\end{eqnarray}

By (\ref{eq3}), we have $\bar b_{i1}=-b_{i1},~\bar b_{l1}=-b_{l1}$. Because $SRB$ and $SR\bar B$ are skew-symmetric, we have $s_1r_1b_{1l}=-b_{l1}s_lr_l, ~s_1r_1\bar b_{1l}=-\bar b_{l1}s_lr_l$ , thus  $s_1r_1[b_{1l}]_+=[-b_{l1}]_+s_lr_l$,  and
 $$M_{il}=\begin{cases}-\bar b_{il}~,&\text{if } i=1\text{ or }l=1,\\ \bar b_{il}+ r_1\left(\bar b_{i1}[-\bar b_{1l}]_++[\bar b_{i1}]_+\bar b_{1l}\right)~,&\text{otherwise}.
  \end{cases}$$
So, by mutation of $\bar B$, we have $M=\mu_1(\bar B)$, but $\mu_1(\bar B)=B$, so we get $M=B$.
\end{proof}

Let $\epsilon\in\{+,-\},~E_k^{\epsilon}=(e_{il})_{n\times n},~e_{il}=\begin{cases} \delta_{il}~,&  \text{if } l\neq k,\\ -1~, & \text{if }i=l=k, \\ [\epsilon b_{ik}]_{+}r_k~,&\text{otherwise}.\end{cases}$ We know $(E_k^{\epsilon})^2=I_n.$

\begin{Corollary}
\label{3cor1}Keep the above notations. For a seed $\Sigma$ of $M$, if $\bar \Sigma=\mu_k(\Sigma)$, then $E_k^{\epsilon}\bar BR^{-1}S^{-1}{E_k^{\epsilon}}^\top=BR^{-1}S^{-1}.$

\begin{proof}
Let $I_k^{+}=\{i~|b_{ik}\geq 0\}$, and $H_0=H|_{x_i=t,~\forall i\in I_k^+ }$. We know
\begin{eqnarray}
\lim\limits_{t\rightarrow +\infty} \left(([-b_{ik}]_+r_k+\frac{\sum\limits_{m=0}^{r_k}mz_{k,m}\hat y_k^mb_{ik}}{\sum\limits_{m=0}^{r_k}z_{k,m}\hat y_k^m})|_{x_i=t,~\forall i\in I_k^+}\right)&=&[b_{ik}]_+r_k\nonumber\\
\lim\limits_{t\rightarrow +0} \left(([-b_{ik}]_+r_k+\frac{\sum\limits_{m=0}^{r_k}mz_{k,m}\hat y_k^mb_{ik}}{\sum\limits_{m=0}^{r_k}z_{k,m}\hat y_k^m})|_{x_i=t,~\forall i\in I_k^+}\right)&=&[-b_{ik}]_+r_k.\nonumber
\end{eqnarray}

Thus $\lim\limits_{t\rightarrow +\infty} H_0=E_k^+,~\lim\limits_{t\rightarrow 0} H_0=E_k^-$.
By Lemma \ref{3lem1},  we have $E_k^{\epsilon}\bar BR^{-1}S^{-1}{E_k^{\epsilon}}^\top=BR^{-1}S^{-1}.$

\end{proof}

\end{Corollary}
\begin{Remark}
By corollary \ref{3cor1},  we obtain $\mu_k(B)=E_k^{\epsilon}BR^{-1}S^{-1}(E_k^{\epsilon})^\top SR$, which was proved for standard cluster pattern in \cite{BZ}.

\end{Remark}

 Let $t_0,~t$ be two vertices in $\mathbb T_n$ with a walk  $$t_0^{~\underline{  \quad k_1\quad   }}~ t_1^{~\underline{\quad k_2 \quad}} ~t_2^{~\underline{\quad k_3 \quad}} ~\cdots ~t_{m-1} ^{~\underline{~\quad k_{m} \quad}}~ t_m=t$$  connecting $t_0$ and $t$ in $\mathbb T_n$. Write $\Sigma_{t_j}$  the seed corresponding to $t_j,~j=1,2,\cdots,m$.

Now we can give the useful formula as follows:

\begin{Theorem}
\label{3thm1}({\bf Cluster Formula})~ Keep the above notations. It holds that $$H_{t_0}^{t}(X)(B_tR^{-1}S^{-1}) H_{t_0}^{t}(X)^{\top}=B_{t_0}R^{-1}S^{-1}, \;\;\;\;\;\;\; \text{and}\;\;\;\;\;\;\; detH_{t_0}^{t}(X)=(-1)^m.$$

\begin{proof}
By Lemma \ref{3lem2}, $H_{t_0}^{t}(X)=H_{t_0}^{t_m}(X)=H_{t_0}^{t_1}(X)H_{t_1}^{t_2}(X)\cdots H_{t_{m-1}}^{t_m}(X)$. By Lemma \ref{3lem1},
\begin{eqnarray}
&&H_{t_0}^{t}(X)(B_tR^{-1}S^{-1}) H_{t_0}^{t}(X)^{\top}\nonumber\\
&=&H_{t_0}^{t_1}(X)H_{t_1}^{t_2}(X)\cdots H_{t_{m-1}}^{t_m}(X)(B_{t_m}R^{-1}S^{-1})H_{t_{m-1}}^{t_m}(X)^{\top}
H_{t_{m-2}}^{t_{m-1}}(X)^{\top}\cdots H_{t_0}^{t_1}(X)^{\top}\nonumber\\
&=&H_{t_0}^{t_1}(X)H_{t_1}^{t_2}(X)\cdots H_{t_{m-2}}^{t_{m-1}}(X)(B_{t_{m-1}}R^{-1}S^{-1})H_{t_{m-2}}^{t_{m-1}}(X)^{\top}
H_{t_{m-3}}^{t_{m-2}}(X)^{\top}\cdots H_{t_0}^{t_1}(X)^{\top}\nonumber\\
&=&\cdots=H_{t_0}^{t_1}(X)(B_{t_1}R^{-1}S^{-1})H_{t_0}^{t_1}(X)^{\top}=B_{t_0}R^{-1}S^{-1}.\nonumber
\end{eqnarray}
And, $detH_{t_0}^{t}(X)=detH_{t_0}^{t_1}(X)detH_{t_1}^{t_2}(X)\cdots detH_{t_{m-1}}^{t_{m}}(X)=(-1)^m.$
\end{proof}

\end{Theorem}

\begin{Corollary}
\label{3cor2} $rank(B_t)=rank(B_{t_0})$ and $det(B_t)=det(B_{t_0})$.

\end{Corollary}

\begin{Corollary}
\label{cor3}If $X_t=X_{t_0}$, then $B_t=B_{t_0}$.

\begin{proof}
If  $X_t=X_{t_0}$, then $x_{i;t}=x_{i;t_0}$. In this case, $H_{t_0}^{t}(X)=I_n$, so we have $B_t=B_{t_0}$.
\end{proof}
\end{Corollary}

\begin{Example}
Consider $S=diag\{1,2\}$,$R=diag\{2,1\},~Z=(z_{1,1})$, assume that $M$ is an $(R,Z)$-cluster pattern via initial seed $\Sigma_{t_0}=(X_{t_0}, Y_{t_0}, B_{t_0})$ with $R$-skew-balance $S$, where $X_{t_0}=(x_1,x_2),~Y_{t_0}=(y_1,y_2)$, $B_{t_0}=\begin{pmatrix}0&-1\\1&0\end{pmatrix}$. Let $\Sigma_t=\mu_2\mu_1(\Sigma_{t_0})$, we have $$x_{1;t}=\frac{1+z\hat y_1+\hat y_2^2}{x_1(1\oplus zy_1\oplus y_1^2)},\;\;~x_{2;t}=\frac{1+\hat y_2+z\hat y_1\hat y_2+\hat y_1^2\hat y_2}{x_2(1\oplus y_2\oplus zy_1y_2\oplus y_1^2y_2)},\;\;~B_t=B_{t_0},$$ where $\hat y_1=y_1x_2,~ \hat y_2=y_2x_1^{-1}$. Thus
$$H_{t_0}^{t}(X)=\begin{pmatrix}-1&\frac{-y_2-zy_1y_2x_2-y_1^2y_2x_2^2}{x_1+y_2+zy_1y_2x_2+y_1^2y_2x_2^2}\\ \frac{zy_1x_2+2y_1^2x_2^2}{1+zy_1x_2+y_1^2x_2^2}&\frac{y_1^2y_2x_2^2-x_1-y_2}{x_1+y_2+zy_1y_2x_2+y_1^2y_2x_2^2} \end{pmatrix}.$$
 It is easy to check that $H_{t_0}^{t}(X)(B_tR^{-1}S^{-1}) H_{t_0}^{t}(X)^{T}=B_{t_0}R^{-1}S^{-1}$ and $det(H_{t_0}^{t}(X))=1$.

More information on this example can be seen at Example 2.3 of \cite{NT}.

\end{Example}

\subsection{Connection between cluster formula and compatible 2-forms}.

In \cite{GSV1,GSV}, M. Gekhtman, M. Shapiro and A. Vainshtein defined a closed differential 2-form $\omega$ compatible with a skew-symmerizable cluster algebra and proved such 2-form always exists for a cluster algebra of geometric type if its exchange matrices  have no zero rows. In this part, we give the connection between compatible 2-forms and the cluster formula, and in particular, we prove the compatible 2-form always exists for any $(R,Z)$ cluster pattern.

\begin{Definition}
A closed rational differential 2-form $\omega$ on an $n$-affine space is {\bf compatible with} the $(R,Z)$-cluster pattern $M$ if for any cluster $X=(x_1,\cdots, x_n)$ one has $\omega=\sum_{i=1}^n \sum_{j=1}^n w_{ij}\frac{d x_i}{x_i}\wedge \frac{d x_j}{ x_j}$, with $w_{ij}\in\mathbb Q$. The matrix $\Omega=(w_{ij})$ is called the {\bf coefficient matrix} of $\omega$ with respect to $X$.
\end{Definition}

Trivially, the coefficient matrix $\Omega$ is skew-symmetric.

\begin{Theorem}
Let $M$ be an $(R,Z)$-cluster pattern with any coefficients.

(i)~ A closed rational differential 2-form $\omega$ on an n-affine space is compatible with $M$ if and only if there exists a family of skew-symmetric matrices $\{\Omega_t\in\mathbb Q^{n\times n})|t\in\mathbb T_n\}$ such that for any $t_0,t\in\mathbb T_n$, we have $H_{t_0}^{t}(X)\Omega_t H_{t_0}^{t}(X)^{\top}=\Omega_{t_0}$.

(ii)~ In particular, there always exists  a closed rational differential 2-form $\omega$ compatible with $M$.
\end{Theorem}

\begin{proof} (i): ~~
``$\Longrightarrow$": Assume that $\omega$ is a closed rational differential 2-form on the n-affine space compatible with $M$, $\Omega_t=(w_{ij}^t)$ is the coefficient matrix of $\omega$ with respect to $X_t$.
We know $dx_{i;t}=\sum_{k=1}^n \frac{\partial x_{i;t}}{\partial x_{k;t_0}}d x_{k;t_0}$, thus $\frac{d x_{i;t}}{x_{i;t}}=\sum_{k=1}^n  \frac{1}{x_{i;t}}\cdot \frac{\partial x_{i;t}}{\partial x_{k;t_0}}d x_{k;t_0}$, and

\begin{eqnarray}
\sum_{i=1}^n \sum_{j=1}^n w_{ij}^t\frac{d x_{i;t}}{x_{i;t}}\wedge \frac{d x_{j;t}}{x_{j;t}}&=&\sum_{i=1}^n \sum_{j=1}^n w_{ij}^t(\sum_{k=1}^n  \frac{1}{x_{i;t}}\cdot \frac{\partial x_{i;t}}{\partial x_{k;t_0}}d x_{k;t_0})\wedge(\sum_{l=1}^n  \frac{1}{x_{j;t}}\cdot \frac{\partial x_{j;t}}{\partial x_{l;t_0}}d x_{l;t_0})\nonumber\\
&=&\sum_{i=1}^n \sum_{j=1}^n
\sum_{k=1}^n \sum_{l=1}^n w_{ij}^t\frac{x_{k;t_0}}{x_{i;t}} \frac{\partial x_{i;t}}{\partial x_{k;t_0}}\cdot\frac{x_{l;t_0}}{x_{j;t}} \frac{\partial x_{j;t}}{\partial x_{l;t_0}}\frac{dx_{k;t_0}}{x_{k;t_0}}\wedge\frac{dx_{l;t_0}}{x_{l;t_0}}\nonumber\\
&=&\sum_{k=1}^n \sum_{l=1}^n (\sum_{i=1}^n \sum_{j=1}^nw_{ij}^t\frac{x_{k;t_0}}{x_{i;t}}\frac{\partial x_{i;t}}{\partial x_{k;t_0}}\cdot\frac{x_{l;t_0}}{x_{j;t}}\frac{\partial x_{j;t}}{\partial x_{l;t_0}})\frac{dx_{k;t_0}}{x_{k;t_0}}\wedge\frac{dx_{l;t_0}}{x_{l;t_0}}.\nonumber
\end{eqnarray}

Since
$$\omega=\sum_{i=1}^n \sum_{j=1}^n w_{ij}^t\frac{d x_{i;t}}{x_{i;t}}\wedge \frac{d x_{j;t}}{x_{j;t}}=\sum_{k=1}^n \sum_{l=1}^n w_{kl}^{t_0}
\frac{d x_{k;t_0}}{x_{k;t_0}}\wedge \frac{d x_{l;t_0}}{x_{l;t_0}},$$
we have $$w_{kl}^{t_0}=\sum_{i=1}^n \sum_{j=1}^nw_{ij}^t\frac{x_{k;t_0}}{x_{i;t}}\frac{\partial x_{i;t}}{\partial x_{k;t_0}}\cdot\frac{x_{l;t_0}}{x_{j;t}}\frac{\partial x_{j;t}}{\partial x_{l;t_0}},$$
that is,
$$\Omega_{t_0}=H_{t_0}^{t}(X)\Omega_t H_{t_0}^{t}(X)^{\top}.$$

``$\Longleftarrow$": Assume that $\{\Omega_t\in \mathbb Q^{n\times n}|t\in\mathbb T_n\}$ is a set of skew-symmetric matrices, satisfying $$H_{t_0}^{t}(X)\Omega_t H_{t_0}^{t}(X)^{\top}=\Omega_{t_0},$$ for any $t_0,t\in\mathbb T_n$, let $\omega=\sum_{i=1}^n \sum_{j=1}^n w_{ij}^{t_0}\frac{d x_{i;t_0}}{x_{i;t_0}}\wedge \frac{d x_{j;t_0}}{ x_{j;t_0}},$ be a closed rational differential 2-form on an $n$-affine space.
 Replacing $dx_{i;t_0}=\sum_{k=1}^n \frac{\partial x_{i;t_0}}{\partial x_{k;t}}d x_{k;t}$ into $\omega$, we can see that $\omega$ is compatible with $M$.

(ii):~~
By Theorem \ref{3thm1}, $\{B_tR^{-1}S^{-1}\in\mathbb Q^{n\times n}|t\in\mathbb T_n\}$ is a family of skew symmetric matrices satisfying $H_{t_0}^{t}(X)(B_tR^{-1}S^{-1})H_{t_0}^{t}(X)^{\top}=B_{t_0}R^{-1}S^{-1}.$ Let $\Omega_t=B_tR^{-1}S^{-1}$. Then by (i),   there always exists  a closed rational differential 2-form $\omega$ compatible with $M$.
\end{proof}

\section{Answer to Conjecture \ref{conj} for generalized cluster algebras}

\subsection{ On Conjecture \ref{conj}(b) in case of weak geometric type }.

In this section, we firstly prove Conjecture \ref{conj} ($b$) for the generalized cluster patterns with coefficients of weak geometric type by using  the cluster formula (see Theorem \ref{4cor2}). The corresponding result for $(R,Z)$-cluster pattern with coefficients in general
 semmifield $\mathbb P$, will be studied in the second part of this section, using the theory of $D$-matrix pattern. Before proving Theorem \ref{4cor2}, we need some preparations.

\begin{Theorem}
 (Theorem 2.5 of \cite{CS})~ For any $(R,Z)$-cluster pattern with coefficients in $\mathbb P$, each cluster variable $x_{i;t}$ can be expressed as a Laurent polynomial in $\mathbb {ZP}[X_{t_0}^{\pm1}]$.

\end{Theorem}

\begin{Definition}
 Let $M$ be an $(R,Z)$-cluster pattern with principle coefficients at $t_0$, by the Laurent property, each cluster variable $x_{i;t}$ is expressed as a Laurent polynomial $X_{i;t}\in\mathbb {ZP}[X_{t_0}^{\pm1}]$, called the {\bf X-function} of $x_{i,t}$, where $\mathbb P=Trop(Y_{t_0},Z)$.

\end{Definition}

\begin{Definition}
The {\bf $F$-polynomial} of $x_{i;t}$ is defined by $F_{i;t}=X_{i;t}|_{x_{1;t_0}=\cdots x_{n;t_0}=1}\in\mathbb Z[Y_{t_0},Z]$.

\end{Definition}
\begin{Proposition}
\label{3pro1}(Proposition 3.3 of \cite{NT}) We have $X_{i;t}\in\mathbb Z[X_{t_0}^{\pm1},Y_{t_0},Z]$.

\end{Proposition}

\begin{Proposition}
\label{4pro1}(Proposition 3.19 and Theorem 3.20 of \cite{NT}) Each {\bf $F$-polynomial} has constant term 1.

\end{Proposition}

\begin{Corollary}
\label{4cor3}Let $M$ be an $(R,Z)$-cluster pattern with principle coefficients at $t_0$, if $x$ is a cluster variable in $M$, then $-x$ can not be a cluster variable in $M$.
\end{Corollary}
\begin{proof}
If both $x$ and $-x$ are cluster variables in $M$, assume that $F$ is the {\bf$F$-polynomial} corresponding to $x$, then $-F$ is the {\bf$F$-polynomial} corresponding to $-x$. This will contradict to that each {\bf$F$-polynomial} has constant term 1.
\end{proof}

Let $M$ be an $(R,Z)$-cluster pattern with principle coefficients and initial seed $\Sigma=(X,Y,B)$, the author in \cite{NT}
introduced a $\mathbb Z^n$-grading of $\mathbb Z[X^{\pm1},Y,Z]$ as follows:
$$deg(x_i)={\bf e}_i,~~deg(y_i)=-{\bf b}_i,~~deg(z_{i;m})=0,$$
where $\bf{e}_i$ is the $i$th column vector of $I_n$, and $\bf{b}_i$ is the $i$th column vector of the initial exchange matrix $B$.
Note that these degrees are vectors in $\mathbb Z^n$.

In \cite{NT}, the author proved that the $X$-functions are homogeneous with respect to the $\mathbb Z^n$-grading, and thanks to this, the $\bf g$-{\bf vector} of a
cluster variable $x_{i;t}$ is defined to be the degree of its $X$-function $X_{i;t}$. From this definition, we have $deg(X_{i;t})=(g_{1i}^t,~g_{2i}^t,~\cdots,~g_{ni}^t)^{\top}\in\mathbb Z^n$.

\begin{Theorem}
\label{4thm1}(Theorem 3.22 and Theorem 3.23 of \cite{NT}) Let $M$ be an $(R,Z)$-cluster pattern with coefficients in $\mathbb P$ and initial seed at $t_0$, $M_{pr}$ be the corresponding $(R,Z)$-cluster pattern with initial principle coefficients at $t_0$, which has the same initial cluster and exchange matrix with $M$, then for the cluster variables $x_{i;t}$ and the coefficients $y_{i;t}$ of $M$ at $t$, it holds that
\begin{eqnarray}
y_{i;t}&=&\prod\limits_{j=1}^ny_{j;t_0}^{c_{ji}^t}\prod\limits_{j=1}^n\left(F_{j;t}|_{\mathbb P(Y_{t_0},Z)}\right)^{b_{ji}^t};\\
x_{i;t}&=&\left(\prod\limits_{j=1}^nx_{j;t_0}^{g_{ji}^t}\right)\frac{F_{i;t}|_{\mathcal F}(\hat Y_{t_0},Z)}{F_{i;t}|_{\mathbb P}(Y_{t_0},Z)}.
\end{eqnarray}
\end{Theorem}

\begin{Proposition}
\label{thmg}Assume that $M$ is an $(R,Z)$-cluster pattern with principle coefficients at $t_0, $ and $S$ is the $R$-skew-balance of $M$. Let $ G_t=(g_{ij}^t)_{n\times n}$. Then we have $G_t=H_{t_0}^{t}(X)|_{Y_{t_0}=0}$.
\end{Proposition}
\begin{proof}
By Theorem \ref{4thm1}, we know
\begin{eqnarray}
x_{j;t}&=&F_{j;t}(\hat Y_{t_0},Z)x_{1;t_0}^{g_{1j}^t}\cdots x_{n;t_0}^{g_{nj}^t},\nonumber\\
 \frac{\partial x_{j;t}}{\partial x_{i;t_0}}&=&\frac{g_{ij}^t}{x_{i;t_0}}(x_{1;t_0}^{g_{1j}^t}\cdots x_{n;t_0}^{g_{nj}^t})F_{j;t}(\hat Y_{t_0},Z)+x_{1;t_0}^{g_{1j}^t}\cdots x_{n;t_0}^{g_{nj}^t}\sum_{k=1}^n \frac{\partial F_{j;t}(\hat Y_{t_0},Z) }{y_k} \frac{b_{ik}^{t_0}}{x_{t_0}}\hat y_{k}.\nonumber\\
\text{Thus}\hspace{5mm} \frac{x_{i;t_0}}{x_{j;t}}\frac{\partial x_{j;t}}{\partial x_{i;t_0}}&=&g_{ij}^t+\sum_{k=1}^n\frac{b_{ik}^{t_0}\hat y_j}{F_{j;t}(\hat Y_{t_0},Z)}\frac{\partial F_{j;t}(\hat Y_{t_0},Z) }{y_k},\nonumber
\end{eqnarray}
By proposition \ref{4pro1}, it is to make sense to take $Y_{t_0}=0$ in above equation, so we have
$$g_{ij}^t=\frac{x_{i;t_0}}{x_{j;t}}\frac{\partial x_{j;t}}{\partial x_{i;t_0}}|_{Y_{t_0}=0},$$
i.e. $G_t=H_{t_0}^{t}(X)|_{Y_{t_0}=0}$.
\end{proof}

From this result and the cluster formula, it is easy to see that  $G_t(B_tR^{-1}S^{-1})G_t^{\top}=B_{t_0}R^{-1}S^{-1}$ and $det(G_t)=\pm1$,
 as obtained in \cite{NT}.

Assume that $M$ is an $(R,Z)$-cluster pattern of weak geometric type with initial seed $(X_{t_0},Y_{t_0},B_{t_0})$. Using the notations in Proposition \ref{2pro1} and assuming $S$ an $R$-skew-balance of $M$, we define $\tilde M$  a $(\tilde R,\tilde Z)$-cluster pattern with $\tilde R$-skew-balance $\tilde S$ and trivial coefficients, given by $(\tilde X_{t_0},\tilde B_{t_0})$, where
$\tilde R=diag\{R,I_h\},\tilde Z=Z,$ $\tilde X_{t_0}=(\tilde x_{1;t_0},\cdots,\tilde x_{h+n;t_0})$, with $\tilde x_{i;t_0}=x_{i;t_0}$, $\tilde x_{n+j;t_0}=y_{j;t_0}$ for $i=1,\cdots,n,~j=1,2,\cdots,h$,$~ \tilde S=diag\{S,I_h\}$, $\tilde B_{t_0}=\begin{pmatrix}B_{t_0}&-R^{-1}S^{-1}C_{t_0}^{\top}\\C_{t_0}&0\end{pmatrix}$.

Clearly, $M$ is a restriction of $\tilde M$ from $\{1,\cdots,h+n\}$ to $\{1,\cdots,n\}$ at $t_0$.
Assume that $\Sigma_t=\mu_{k_m}\cdots\mu_{k_2}\mu_{k_1}(\Sigma_{t_0})$, where $1\leq k_j\leq n, j=1,2,\cdots,m$, then  $\tilde \Sigma_t=\mu_{k_m}\cdots\mu_{k_2}\mu_{k_1}(\tilde\Sigma_{t_0})$.
Since $1\leq k_j\leq n, j=1,2,\cdots,m$, we can write $H_{t_0}^{t}(\tilde X)=\begin{pmatrix} H_{t_0}^{t}(X)&0\\H_t&I_h\end{pmatrix}$, $\tilde B_t=\begin{pmatrix} B_t&-S^{-1}R^{-1}C_t^{\top}\\C_t&\bar B_t\end{pmatrix}$.

\begin{Proposition}
\label{4cor1} Keep the above notations, it holds $SR(H_tB_t+C_t)R^{-1}S^{-1}H_{t_0}^{t}(X)^{\top}=C_{t_0}$.
\end{Proposition}

\begin{proof}
By Theorem \ref{3thm1}, we have $H_{t_0}^{t}(\tilde X)(\tilde B_t\tilde R^{-1}\tilde S^{-1})H_{t_0}^{t}(\tilde X)^{\top}=\tilde B_{t_0}\tilde R^{-1}\tilde S^{-1}$,
thus
 \begin{eqnarray}
 &&\begin{pmatrix} H_{t_0}^{t}(X)&0\\H_t&I_h\end{pmatrix} \begin{pmatrix} B_t&-S^{-1}R^{-1}C_t^{\top}\\C_t&\bar B_t\end{pmatrix} \begin{pmatrix} R^{-1}&0\\0&I_h\end{pmatrix}\begin{pmatrix} S^{-1}&0\\0&I_h\end{pmatrix}\begin{pmatrix} H_{t_0}^{t}(X)&0\\H_t& I_h \end{pmatrix}^{\top} \nonumber\\
 &=&\begin{pmatrix} B_{t_0}&-R^{-1}S^{-1}C_{t_0}^{\top}\\C_{t_0}&0\end{pmatrix}\begin{pmatrix} R^{-1}&0\\0&I_h\end{pmatrix}\begin{pmatrix} S^{-1}&0\\0&I_h\end{pmatrix}\nonumber.
 \end{eqnarray}
So we have $SR(H_tB_t+C_t)R^{-1}S^{-1}H_{t_0}^{t}(X)^{\top}=C_{t_0}$.
\end{proof}

\begin{Remark}
\label{3rmk1} By Theorem \ref{3thm1}, $H_{t_0}^{t}(X)B_tR^{-1}S^{-1}H_{t_0}^{t}(X)^{\top}=B_{t_0}R^{-1}S^{-1},\hspace{1mm} det(H_{t_0}^{t}(X))=\pm1$. Then using this proposition, we can obtain $$C_t=R^{-1}S^{-1}C_{t_0}(H_{t_0}^{t}(X)^{\top})^{-1}SR-H_t(H_{t_0}^{t}(X))^{-1}B_{t_0}R^{-1}S^{-1}(H_{t_0}^{t}(X)^{\top})^{-1}SR.$$
\end{Remark}

Using Proposition \ref{4cor1}, the following result in \cite{NT} can be given directly.
\begin{Corollary} (\cite{NT})~
Let $M$ be an $(R,Z)$-cluster pattern with principle coefficients at $t_0$ and $R$-skew-balance $S$, then $$SRC_tR^{-1}S^{-1}G_t^{\top}=I_n.$$
\begin{proof}
By the definition of $H$-matrix, and proposition \ref{4pro1}, we know $H_t|_{Y_{t_0}=0}=0$.
By proposition \ref{4cor1} and proposition \ref{thmg}, we have $SRC_tR^{-1}S^{-1}G_t^{\top}=I_n.$
\end{proof}

\end{Corollary}

Now, we can give the positive affirmation on Conjecture \ref{conj} (b) in case of weak geometric type.

\begin{Theorem}
\label{4cor2}Assume that $M$ is an $(R,Z)$-cluster pattern of weak geometric type at $t_0$, with an $R$-skew-balance $S$.  Then for each $t$, the seed $\Sigma_t$ is uniquely determined by $X_t$.
\end{Theorem}

\begin{proof}
We know $\Sigma_t=(X_t,Y_t,B_t)$ and $Y_t$ is uniquely determined by $C_t$. However, by Theorem \ref{3thm1}, we have
\begin{eqnarray}
\label{eqb}B_t=(H_{t_0}^{t}(X))^{-1}B_{t_0}R^{-1}S^{-1}(H_{t_0}^{t}(X)^{\top})^{-1}SR.
\end{eqnarray}
By remark \ref{3rmk1}, we have
\begin{eqnarray}
\label{eqc}C_t=R^{-1}S^{-1}C_{t_0}(H_{t_0}^{t}(X)^{\top})^{-1}SR-H_t(H_{t_0}^{t}(X))^{-1}B_{t_0}R^{-1}S^{-1}(H_{t_0}^{t}(X)^{\top})^{-1}SR.
\end{eqnarray}
We know the right side of (\ref{eqb}) and (\ref{eqc}) is uniquely determined by $X_t$,
thus $B_t$ and $C_t$ is uniquely determined  by $X_t$, which implies that $\Sigma_t$ is uniquely determined  by $X_t$.
\end{proof}

\subsection{$D$-matrix pattern and answer to Conjecture \ref{conj}}.

Let $M$ be an $(R,Z)$-cluster pattern with coefficients in $\mathbb P$ and initial seed $\Sigma_{t_0}=(X_{t_0},P_{t_0},B_{t_0})$. By Laurent phenomenon, we can express the cluster variable $x_{i;t}$ in $\Sigma_t$, as
\begin{eqnarray}
\label{eqd}x_{i;t}=\frac{f_{i;t}(x_{1;t_0},\cdots,x_{n,t_0})}{x_{1;t_0}^{d_{1i}^t}\cdots x_{n;t_0}^{d_{ni}^t}},
\end{eqnarray}
 where $f_{i;t}$ is a polynomial in $x_{1;t_0},\cdots,x_{n;t_0}$ with coefficients in $\mathbb {ZP}$, such that $x_{j;t_0}\nmid f_{i;t}$.

 Define  ${\bf d}_i^t=(d_{1i}^t,~d_{2i}^t,\cdots,~d_{ni}^t)^{\top}$ which is called the {\bf d-vector} of $x_{i;t}$.

 Define $D_t=(d_{ij}^t)=({\bf d}_1^t,~{\bf d}_2^t,\cdots,~{\bf d}_n^t)$,  called the {\bf$D$-matrix} of the cluster $X_t$.
  Clearly, $D_{t_0}=-I_n$.

\begin{Proposition}\label{mutationAD}
$D_t$ is uniquely determined by the initial condition $D_{t_0}=-I_n$, together with the relation as follows under mutation of seeds:
\begin{equation}\label{mutationD}
(D_{t^{\prime}})_{ij}=\begin{cases}d_{ij}^t  & \text{if } j\neq k;\\ -d_{ik}^t+max\{\sum\limits_{b_{lk}^t>0}d_{il}^t b_{lk}^tr_k, \sum\limits_{b_{lk}^t<0} -d_{il}^tb_{lk}^tr_k\}  &\text{if } j=k.\end{cases}
 \end{equation}
 for any $t,t'\in\mathbb T_n$ with edge $t^{~\underline{\quad k \quad}} ~t^{\prime}$.
\end{Proposition}

\begin{proof}
We know $ x_{k;t^{\prime}}= x_{k;t}^{-1}\left(\prod\limits_{j=1}^nx_{j;t}^{[-b_{jk}^t]_+}\right)^{r_k}\frac{\sum\limits_{m=0}^{r_k}z_{k,m}\hat y_{k;t}^m}{\bigoplus\limits_{m=0}^{r_k}z_{k,m} y_{k;t}^m}$, where $\hat y_{k;t}^m=y_{k;t}\prod\limits_{j=1}^nx_{j;t}^{b_{jk}^t}$.
We can obtain $d_{ik}^{t^{\prime}}=-d_{ik}^t+max\{\sum\limits_{b_{lk}^t>0}d_{il}^t b_{lk}^tr_k, \sum\limits_{b_{lk}^t<0} -d_{il}^tb_{lk}^tr_k\}$, by (\ref{eqd}).
\end{proof}

In this proposition, the case for standard cluster pattern has been given in \cite{FZ3}.

\begin{Corollary}
\label{6pro1}The {\bf d}-vectors of the $(R,Z)$-cluster pattern with initial seed $(X_{t_0},Y_{t_0},B_{t_0})$ coincide with the {\bf d}-vectors of the standard cluster pattern with initial seed $(X_{t_0},Y_{t_0}, B_{t_0}R)$.
\end{Corollary}

\begin{Definition}~
 A {\bf $D$-matrix pattern} $W$ at $t_0$ is an assignment for each pair   $\Delta_t:=(D_t, Q_t)$, called a {\bf matrix seed}, to a vertex $t$ of the $n$-regular tree $\mathbb T_n$ with  $\Delta_{t_0}=(-I_n,Q_{t_0})$, which is called the {\bf initial matrix seed}, where $Q_{t_0}$ is a skew-symmetrizable matrix.
And for any edge $t^{~\underline{\quad k \quad}}~ t^{\prime}$, $\Delta_{t'}=(D_{t^{\prime}}, Q_{t^{\prime}})$ and $\Delta_t=(D_t, Q_t)$ are related with $Q_{t^{\prime}}=\mu_k(Q_t)$ by the standard matrix mutation $\mu_k$ and $D_{t'}$ is defined satisfying (\ref{mutationD}) in  Proposition \ref{mutationAD}. Denote  $\mu_k^{ms}(\Delta_t):=\Delta_{t^{\prime}}$, where $\mu_k^{ms}$ is called the {\bf mutation of matrix seed in the direction $k$}.
\end{Definition}

\begin{Remark}
By Remark \ref{rmkb}, any $(R,Z)$-cluster pattern $M$ with initial seed at $t_0$ can supply the corresponding $D$-matrix pattern $W$ with matrix seed $\Delta_t=(D_t, B_tR)$ at $t\in \mathbb T_n$, where $D_t$ is the $D$-matrix of the cluster $X_t$.  This $W$ is called the {\bf $D$-matrix pattern induced by $M$ at $t_0$}.
\end{Remark}

\begin{Definition}

$(i)$  For an $(R,Z)$-cluster pattern $M$, two seeds $\Sigma_t=(X_{t},Y_{t},B_{t})$ and $\Sigma_{t'}=(X_{t'},Y_{t'},B_{t'})$, or say, their corresponding  vertices $t$ and $t'$ in $\mathbb T_n$, are called {\bf $\mathcal M$-equivalent} if there exists a permutation $\sigma\in S_n$ such that  $x_{i;t^{\prime}}=x_{\sigma(i);t}$, $y_{i;t^{\prime}}=y_{\sigma(i);t}$ and $b_{ij}^{t^{\prime}}=b_{\sigma(i)\sigma(j)}^t$, denote as $\Sigma_t\simeq_{\mathcal M}\Sigma_{t'}$.

$(ii)$ For a $D$-matrix pattern $W$, two matrix seeds $\Delta_t=(D_{t},Q_{t})$ and $\Delta_{t'}=(D_{t'},Q_{t'})$, or say, their corresponding vertices $t$ and $t^{\prime}$ in $\mathbb T_n$ are {\bf$\mathcal W$-equivalent} if there exists a permutation $\sigma\in S_n$ such that $ {\bf d}_{i}^{t^{\prime}}= {\bf d}_{\sigma(i)}^t$ and $q_{ij}^{t^{\prime}}=q_{\sigma(i)\sigma(j)}^t$, denote as $\Delta_t\simeq_{\mathcal W}\Delta_{t'}$.
\end{Definition}
\begin{Definition}
The {\bf exchange graph} $\Gamma$ of a matrix pattern $W$ (respectively, $(R,Z)$-cluster pattern $M$) is defined as the graph whose vertices are the $\mathcal W$-equivalence classes of matrix seeds $[\Delta_t]$ (respectively, $\mathcal M$-equivalence classes of seeds $[\Sigma_t]$) and whose edges given between $[\Delta_{t_1}]$  and  $[\Delta_{t_2}]$ (respectively, $[\Sigma_{t_1}]$ and $[\Sigma_{t_2}]$) for $t_1,t_2\in \mathbb T_n$ if there exists  $k\in\{1,\cdots,n\}$ such that $\mu_k^{ms}(\Delta_{t_1})\in[\Delta_{t_2}]$ (respectively, $\mu_k(\Sigma_{t_1})\in[\Sigma_{t_2}]$).
\end{Definition}
\begin{Remark}
\label{6rmk1}By the definition, the exchange graph of a $D$-matrix pattern only depends on the initial exchange matrix $\Delta_{t_0}$.
\end{Remark}

 Now we discuss the further relationship between an $(R,Z)$-cluster pattern with initial seed $\Sigma_{t_0}$ and the matrix pattern induced by it at $t_0$.

\begin{Lemma}
\label{6lem2}Let $M_{pr}$ be an $(R,Z)$-cluster pattern with principal coefficients at $t_0$. If there exists a permutation $\sigma\in S_n$ such that ${\bf d}_{\sigma(i)}^t={\bf d}_{i}^{t_0}$, where ${\bf d}_i^t$ are the $i$-th columns of $D_t$ for all $i$, then $x_{\sigma(i);t}=x_{i;t_0}$ and $\Sigma_t\simeq_{\mathcal M}\Sigma_{t_0}$.
\end{Lemma}

\begin{proof}
    By Proposition \ref{3pro1}, there exist polynomials $f_1,\cdots,f_n$ in $x_{1;t_0},\cdots,x_{n;t_0}$ with coefficients in $\mathbb Z[Y_{t_0},Z]$,and $x_{j;t_0}\nmid f_i$ for any $i,j$ in order to get $D_t$. But ${\bf d}_{\sigma(i)}^t={\bf d}_{i}^{t_0}$ and  $D_{t_0}=-I_n$ , then
    \begin{equation}\label{equality}
    x_{\sigma(1);t}=x_{1;t_0}f_1,\cdots,x_{\sigma(n);t}=x_{n;t_0}f_n.
    \end{equation}
 So we have $x_{i;t_0}=\frac{x_{\sigma(i);t}}{f_i(x_{1;t_0},\cdots,x_{n;t_0})}$. Conversely,  there exist polynomials $g_1,\cdots,g_n$ in $x_{1;t},\cdots,x_{n;t}$ with coefficients in $\mathbb Z[Y_{t_0},Z]$, and $x_{j;t}\nmid g_i$ for any $i,j$, such that
 \begin{equation}\label{equalityten}
 x_{1;t_0}=\frac{g_1(x_{1;t},\cdots,x_{n;t})}{x_{1;t}^{k_{11}}\cdots x_{n;t}^{k_{n1}}},\cdots,x_{n;t_0}=\frac{g_n(x_{1;t},\cdots,x_{n;t})}{x_{1;t}^{k_{1n}}\cdots x_{n;t}^{k_{nn}}}.
  \end{equation}
  Hence, for any $i$,  $$\frac{g_i(x_{1;t},\cdots,x_{n;t})}{x_{1;t}^{k_{1i}}\cdots x_{n;t}^{k_{ni}}}=x_{i;t_0}=\frac{x_{\sigma(i);t}}{f_i(\frac{g_1(x_{1;t},\cdots,x_{n;t})}{x_{1;t}^{k_{11}}\cdots x_{n;t}^{k_{n1}}},\cdots,\frac{g_n(x_{1;t},\cdots,x_{n;t})}{x_{1;t}^{k_{1n}}\cdots x_{n;t}^{k_{nn}}})},$$
  For the right side of this equality, we can write $x_{i;t_0}$  as that $x_{i;t_0}=x_{1;t}^{\lambda_{1;i}}\cdots x_{n;t}^{\lambda_{n;i}}/h_i(x_{1;t},\cdots,x_{n;t})$, where $h_i$ is a polynomial in $x_{1;t},\cdots,x_{n;t}$ such that $x_{j;t}\nmid h_i$ for $j=1,\cdots,n$. So we have
  $$g_i(x_{1;t},\cdots,x_{n;t})h_i(x_{1;t},\cdots,x_{n;t})=x_{1;t}^{k_{1;i}+\lambda_{1;i}}\cdots x_{n;t}^{k_{n;i}+\lambda_{n;i}},$$
 However, due to $x_{j;t}\nmid g_i$ and $x_{j;t}\nmid h_i$ for $j=1,\cdots,n$, it implies that $g_i=\pm1=h_i$, then from (\ref{equalityten}), we have $x_{i;t_0}=\frac{\pm1}{x_{1;t}^{k_{1i}}\cdots x_{n;t}^{k_{ni}}}$. From this and by the definition of $H_{t}^{t_0}(X)$,
 we can obtain that $H_{t}^{t_0}(X)=(-k_{ij})_{n\times n}$. By Theorem \ref{3thm1}, $det H_{t}^{t_0}(X)= \pm 1$. By Lemma \ref{3lem2},
 $H_{t_0}^{t}(X)=H_{t}^{t_0}(X)^{-1}$.  Then we have  $H_{t_0}^{t}(X)\in M_n(\mathbb Z)$.
 \begin{Lemma}\label{addlem}
  For any $i=1,\cdots,n$, $x_{i;t}$ is a Laurent monomial in $x_{1;t_0},\cdots,x_{n;t_0}$ and  $f_i=\pm 1$.
  \end{Lemma}
  \begin{proof}
  Without loss of generality, we can assume that $i=1$. By (\ref{equality}) and the definition of $H_{t_0}^t(X)_{j\sigma(1)}$, we can get $H_{t_0}^t(X)_{j\sigma(1)}=\delta_{1j}+\frac{x_{j;t_0}}{f_1}\frac{\partial f_1}{\partial x_{j;t_0}}$. Since $H_{t_0}^{t}(X)\in M_n(\mathbb Z)$,  $\frac{x_{j;t_0}}{f_1}\frac{\partial f_1}{\partial x_{j;t_0}}$ is an integer.
   Write $f_1=a_{m}x_{j;t_0}^{m}+a_{m-1}x_{j;t_0}^{m-1}+\cdots+a_{1}x_{j;t_0}+a_0$, where $a_m\neq 0$ and $a_{0},\cdots,a_{m}$ are polynomials of $x_{1;t_0},\cdots,x_{j-1;t_0},x_{j+1;t_0},\cdots,x_{n;t_0}$ with coefficients in $\mathbb Z[Y_{t_0},Z]$, then $\frac{x_{j;t_0}}{f_1}\frac{\partial f_1}{\partial x_{j;t_0}}=\frac{ma_mx_{j;t_0}^m+(m-1)a_{m-1}x_{j;t_0}^{m-1}+\cdots+a_1x_{j;t_0}}{a_{m}x_{j;t_0}^{m}+a_{m-1}x_{j;t_0}^{m-1}
  +\cdots+a_{1}x_{j;t_0}+a_0}$ is an integer. If $m>0$, then $\frac{x_{j;t_0}}{f_1}\frac{\partial f_1}{\partial x_{j;t_0}}=m$ and $a_0=a_1=\cdots=a_{m-1}=0$. So $f_1=a_mx_{j;t_0}^m$, which contradicts to $x_{j;t}\nmid f_1$. Thus $m=0$ and $f_1=a_0$, which is a polynomial of $x_{1;t_0},\cdots,x_{j-1;t_0},x_{j+1;t_0},\cdots,x_{n;t_0}$ with coefficients in $\mathbb Z[Y_{t_0},Z]$. Since $j$ can take value from $1$ to $n$, $f_1$ must be in $\mathbb Z[Y_{t_0},Z]$. By (\ref{equality}), we have $x_{1;t_0}=\frac{x_{\sigma(1);t}}{f_1}$, then by Proposition \ref{3pro1}, $\frac{1}{f_1}\in\mathbb Z[Y_{t_0},Z]$, which means $f_1=\pm 1$.
  \end{proof}

  Return to the proof of Lemma \ref{6lem2}.
 By Lemma \ref{addlem} and (\ref{equality}), we have $x_{\sigma(i);t}=\pm x_{i;t_0}$. Thus $x_{\sigma(i);t}=x_{i;t_0}$ by Corollary \ref{4cor3}. It is easy to see that $b_{\sigma(i)\sigma(j)}^t=b_{ij}^{t_0},~{\bf g}_{\sigma(i)}^t={\bf g}_i^{t_0}$ and ${\bf c}_{\sigma(i)}^t={\bf c}_i^{t_0}$. Thus $\Sigma_t\simeq_{\mathcal M}\Sigma_{t_0}$.
\end{proof}

\begin{Theorem}
\label{6thm1}Let $M$ be an $(R,Z)$-cluster pattern with initial seed $(X_{t_0},~Y_{t_0},~B_{t_0})$ at $t_0$, $W$ be the $D$-matrix pattern induced by $M$ at $t_0$. Let $M_{pr}$ be the corresponding $(R,Z)$-cluster pattern with principal coefficients at $t_0$ of $M$, denote by $\Sigma_{t}^{pr}$ the seed of $M_{pr}$ at $t$. Then,

(i)~ The following statements are equivalent: (1)~ $\Sigma_{t_1}\simeq_{\mathcal M}\Sigma_{t_2}$, (2)~ $\Sigma^{pr}_{t_1}\simeq_{\mathcal M}\Sigma^{pr}_{t_2}$, (3)~  $\Delta_{t_1}\simeq_{\mathcal W}\Delta_{t_2}$;

  (ii)~  $\Gamma_{M}=\Gamma_{M_{pr}}=\Gamma_W$.
\end{Theorem}

\begin{proof} (i): ~ Firstly, we prove (2)$\Longleftrightarrow$(3). ``$\Longrightarrow$": Obviously.

``$\Longleftarrow$":
Since $\Delta_{t_1}\simeq_{\mathcal W}\Delta_{t_2}$, there exists a permutation $\sigma\in S_n$ such that  ${\bf d}_{i}^{t_1}={\bf d}_{\sigma(i)}^{t_2}$ and $b_{ij}^{t_1}=b_{\sigma(i)\sigma(j)}^{t_2}$.
 We always have $\Sigma_{t_0}^{pr}=\mu_{i_k}\cdots\mu_{i_2}\mu_{i_1}(\Sigma_{t_1}^{pr})$ for a series of mutations $\mu_{i_1}, \mu_{i_2}, \cdots, \mu_{i_k}$.
 Let $\Sigma_{t}^{pr}=\mu_{\sigma(i_k)}\cdots\mu_{\sigma(i_2)}\mu_{\sigma(i_1)}(\Sigma_{t_2}^{pr})$, then we have ${\bf d}_{i}^{t_0}={\bf d}_{\sigma(i)}^{t}$ and $b_{ij}^{t_0}=b_{\sigma(i)\sigma(j)}^{t}$.
 By Lemma \ref{6lem2}, we have $x_{\sigma(i);t}^{pr}=x_{i;t_0}^{pr}$  and ${\bf c}_{\sigma(i)}^t={\bf c}_i^{t_0}$.
 Since $\Sigma_{t_1}^{pr}=\mu_{i_1}\mu_{i_2}\cdots\mu_{i_k}(\Sigma_{t_0}^{pr})$ and $\Sigma_{t_2}^{pr}=\mu_{\sigma(i_1)}\mu_{\sigma(i_2)}\cdots\mu_{\sigma(i_k)}\mu_{\sigma(i_1)}(\Sigma_{t}^{pr})$, we obtain $x_{\sigma(i);t_2}^{pr}=x_{i;t_1}^{pr}$, then $b_{\sigma(i)\sigma(j)}^{t_2}=b_{ij}^{t_1}$  and ${\bf c}_{\sigma(i)}^{t_2}={\bf c}_i^{t_1}$, which means $\Sigma_{t_1}^{pr}\simeq_{\mathcal M}\Sigma_{t_2}^{pr}$.

(1)$\Longrightarrow$(3) is also trivial. Now we prove  (2)$\Longrightarrow$(1).

If $\Sigma_{t_1}^{pr}\simeq_{\mathcal M}\Sigma_{t_2}^{pr}$, then there exists a permutation $\sigma\in S_n$ such that $b_{\sigma(i)\sigma(j)}^{t_1}=b_{ij}^{t_2},~ {\bf c}_{\sigma(i)}^{t_1}={\bf c}_{i}^{t_2},~{\bf g}_{\sigma(i)}^{t_1}={\bf g}_{i}^{t_2}$ and $F_{\sigma(i);t_1}=F_{i;t_2}$.
By Theorem \ref{4thm1}, for cluster variables and coefficients of $M$, we have the relations  $x_{\sigma(i);t_1}=x_{i;t_2},~y_{\sigma(i);t_1}=x_{i;t_2}$, then $X_{t_1}=X_{t_2},~Y_{t_1}=Y_{t_2}$ as sets. It follows that $\Sigma_{t_1}$ and $\Sigma_{t_2}$ are $\mathcal M$-equivalent.

(ii) is obtained directly from (i).
\end{proof}

From this theorem, we now can answer Conjecture \ref{conj} (a),(b),(c) for a generalized cluster algebra in the statements (a),(b),(c) respectively as follows.

\begin{Theorem}
\label{6thm2} Given an $(R,Z)$-cluster pattern $M$ with coefficients in $\mathbb P$ and initial seed $\Sigma_{t_0}$, the following statements hold:

(a)~ The exchange graph only depends on the initial exchange matrix $B_{t_0}$.

(b)~ Every seed $\Sigma_t$ in $M$ is uniquely determined by $X_t$.

(c)~ Two clusters are adjacent in the exchange graph $\Gamma_M$ if and only if they have exactly $n-1$ common cluster variables.

\end{Theorem}

\begin{proof}

(a)~ Let $W$ be the $D$-matrix pattern induced by $M$ at $t_0$. By Theorem \ref{6thm1}, $\Gamma_{M}\simeq\Gamma_{W}$. Moreover, the result follows from  Remark \ref{6rmk1}.

(b)~ By Theorem \ref{6thm1}, $\Sigma_t$ is uniquely determined by $\Delta_t=(D_t,B_tR)$. And, by the definition of $D_t$,  $D_t$ is uniquely determined by $X_t$ and the initial seed $\Sigma_{t_0}$.
By  Theorem \ref{3thm1}, $B_t$ is also determined by $X_t$. Then the result follows.

(c)~
$``\Longrightarrow"$: It is clear from the definition of mutation.

$``\Longleftarrow"$:  Assume that $X_{t_1}$ and $X_{t_2}$ are two clusters of $M$ with $n-1$ common cluster variables, we will prove that the matrix seed $\Delta_{t_1}$ and $\Delta_{t_2}$ are adjacent in the exchange graph $\Gamma_W$. Thus, by (b) and Theorem \ref{6thm1}, $X_{t_1}$ and $X_{t_2}$ are adjacent in the exchange graph $\Gamma_M$.

Let $M_{tr}$ be the corresponding $(R,Z)$-cluster pattern with trivial coefficients of $M$, which has the same initial exchange matrix with $M$. The cluster variables $x_{i,t}$ of $M_{tr}$ can be obtained from the corresponding cluster variables $x_{i,t}^o$ of $M$ via valuing their coefficients to $1$. Hence, any pair of equal cluster variables in $X_{t_1}$ and $X_{t_2}$ respectively becomes a pair of equal  cluster variables in $X_{t_1}^{tr}$ and $X_{t_2}^{tr}$ respectively.
Without loss of generality, let $X_{t_1}^{tr}=(x_1,x_2,\cdots,x_n),~X_{t_2}^{tr}=(w_1,x_2,\cdots,x_n),~X_{t_3}^{tr}=\mu_{x_1}(X_{t_1}^{tr})=(\bar x_1,x_2,\cdots,x_n)$.

If $x_1=w_1$, then by (b) and Theorem \ref{6thm1}, $\Delta_{t_1}\simeq_{\mathcal W}\Delta_{t_1}$ and then $X_{t_1}=X_{t_2}$ as sets. This is a contradiction.
Hence, we have $x_1\not=w_1$. Then, the given clusters $X_{t_1}^{tr}$ and $X_{t_2}^{tr}$ of $M_{tr}$ have also $n-1$ common cluster variables.

By the definition of $H_{t_1}^{t_2}(X^{tr})$ and $H_{t_3}^{t_2}(X^{tr})$, they can be written as the form $H_{t_1}^{t_2}(X^{tr})=\begin{pmatrix}a&O_{1\times (n-1)}\\\alpha_{(n-1)\times 1}&I_{n-1}   \end{pmatrix}$ and $H_{t_3}^{t_2}(X^{tr})=\begin{pmatrix}\bar a&O_{1\times (n-1)}\\ \bar\alpha_{(n-1)\times 1}&I_{n-1}   \end{pmatrix}$. By the cluster formula,  $det H_{t_3}^{t_2}(X^{tr})=-det H_{t_1}^{t_2}(X^{tr})=\pm 1$, thus $\bar a=-a=\pm1$.

If $a=1$, i.e. $\frac{x_1}{w_1}\frac{\partial w_1}{\partial x_1}=1$, we will show that $w_1=x_1$, which is a contradiction.

By Laurent phenomenon, $w_1$ can be written as
\begin{equation}\label{wx}
w_1=f(x_1,\cdots,x_n)/(x_1^{d_1}\cdots x_n^{d_n}),
\end{equation}
 where $f$ is a polynomial in $x_1,\cdots,x_n$ with coefficients in $\mathbb Z[Z]$ with  $x_i\nmid f$ for any $i$. By $1=\frac{x_1}{w_1}\frac{\partial w_1}{\partial x_1}=-d_1+\frac{x_1}{f}\frac{\partial f}{\partial x_1}$, we know $\frac{x_1}{f}\frac{\partial f}{\partial x_1}$ is an integer. Just as in  the proof of Lemma \ref{addlem}, we can show that $x_1$ does not appear in $f$, i.e. $f$ is a polynomial in $x_2,\cdots, x_n$. Thus $\frac{x_1}{f}\frac{\partial f}{\partial x_1}=0$ and $d_1=-1$. Then by (\ref{wx}), we have $x_1=\frac{w_1}{x_2^{-d_1}\cdots x_n^{-d_n}f(x_2,\cdots,x_n)}$. By Laurent phenomenon, it is easy to know $f$ is a monomial in $x_2,\cdots, x_n$. Since $x_i\nmid f$, we obtain $f=1$ and $w_1=x_1x_2^{-d_2}\cdots x_n^{-d_n}$.

For $i\not=1$, if $d_i>0$, we consider the cluster variable $\bar x_i$ obtained from $X_{t_1}^{tr}$ by mutation at cluster variable $x_i$.  Then $x_i\bar x_i=g(x_1,\cdots,x_{i-1},x_{i+1},\cdots,x_{n})$. Clearly, $g$ is a nontrivial  polynomial, otherwise, the generalized cluster algebra generated by $x_i$ will split off, which contradicts to that $d_i\neq 0$. Therefore $w_1=(x_1x_2^{-d_2}\cdots x_{i-1}^{-d_{i-1}}\bar x_i^{d_i}x_{i+1}^{-d_{i+1}}\cdots x_n^{-d_n})/g^{d_i}$.  It contradicts to Laurent phenomenon since $g$ is an exchange polynomial.
 So $d_i\leq0$ for $i=2,\cdots,n$.

 Consider $x_1=w_1x_2^{d_2}\cdots x^{d_n}$ and use the similar discussion as above, we can show $d_i\geq0$ dually for $i=2,\cdots,n$. Thus $d_2=\cdots=d_n=0$, and we obtain $w_1=x_1$. It is impossible. Therefore we have only $a=-1$, then $\bar a=1$.

 Since $\bar a=1$, we can repeat  the above discussion via replacing $X_{t_1}^{tr}$ by $X_{t_3}^{tr}$, and obtain $w_1=\bar x_1$, i.e.  $\mu_{x_1}(X_{t_1}^{tr})=X_{t_2}^{tr}$. By (b) and the definition of matrix seed, it follows that the matrix seed $\Delta_{t_1}$ and $\Delta_{t_2}$ are adjacent in the exchange graph $\Gamma_W$. Then the result holds.
\end{proof}

We know that a pattern (cluster pattern or matrix pattern) is said to be of {\bf finite type}, if the exchange graph has finite many vertexes.

\begin{Corollary} \label{corio}
Assume $M$ is an $(R,Z)$-cluster pattern with initial seed $(X_{t_0},Y_{t_0},B_{t_0})$, let $\bar M$ be the standard cluster pattern   with initial seed $(X_{t_0},Y_{t_0},B_{t_0}R)$ at $t_0$ induced from $M$, then

(i) $\Sigma_{t_1}\simeq_{\mathcal M}\Sigma_{t_2}$ if and only if  $\bar\Sigma_{t_1}\simeq_{\mathcal M}\bar\Sigma_{t_2}$,

(ii) $\Gamma_M\simeq\Gamma_{\bar M}$.

\end{Corollary}
\begin{proof}

We know that $M$ and $\bar M$ induce the same $D$-matrix pattern $W$ at $t_0$. By Theorem \ref{6thm1}, we have
$\Sigma_{t_1}\simeq_{\mathcal M}\Sigma_{t_2}$ if and only if $\Delta_{t_1}\simeq_{\mathcal M}\Delta_{t_2}$ if and only if $\bar\Sigma_{t_1}\simeq_{\mathcal M}\bar\Sigma_{t_2}$, then $\Gamma_M\simeq\Gamma_{W}\simeq\Gamma_{\bar M}$.
\end{proof}

Following Corollary  \ref{corio} (ii), we have furthermore:
\begin{Corollary}\label{final}
$M$ is of finite type if and only if $\bar M$ is of finite type.

\end{Corollary}
\begin{Remark}
The classifications of standard cluster algebras and generalized cluster algebras of finite type has been given respectively in \cite{FZ1} and \cite{CS}. Corollary \ref{final} actually supplies a simple way to give the classification of generalized cluster algebras of finite type via that of standard cluster algebras.
\end{Remark}

{\bf Acknowledgements:}\; This project is supported by the National Natural Science Foundation of China (No.11671350 and No.11571173) and the Zhejiang Provincial Natural Science Foundation of China (No.LZ13A010001).

\end{CJK*}

\end{document}